\documentclass[12pt, reqno]{amsart}
\usepackage{amsmath, amsthm, amscd, amsfonts, amssymb, graphicx, color}
\usepackage[bookmarksnumbered, colorlinks, plainpages]{hyperref}

\def\authorsaddresses#1{\dedicatory{#1}}

\newtheorem{theorem}{Theorem}[section]
\newtheorem{lemma}[theorem]{Lemma}
\newtheorem{proposition}[theorem]{Proposition}

\theoremstyle{definition}

\numberwithin{equation}{section}

\begin{document}
\setcounter{page}{1}

\title[Tetravalent half-arc-transitive graphs of order $12p$]{Tetravalent half-arc-transitive graphs of order $12p$}

\author[ M. Ghasemi, A.A.  Talebi and    N. Mehdipoor]{ M. Ghasemi$^1$,   N. Mehdipoor$^2$ and A.A.  Talebi$^3$    }

\authorsaddresses {$^1$ Department of Mathematics, Urmia University, Urmia 57135, Iran.  email:  m.ghasemi@urmia.ac.ir\\
$^2$ Department of Mathematics, Mazandaran University, email:  nargesmehdipoor@yahoo.com
 \\$^3$  Department of Mathematics, Mazandaran University, email: a.talebi@umz.ac.ir}

\begin{abstract}
A graph is half-arc-transitive if its automorphism group acts transitively
on its vertex set, edge set, but not its arc set. In this paper, we study  all tetravalent half-arc-transitive graphs of
order $12p$.
 \end{abstract}
\maketitle

\section{Introduction }
In this study, all graphs considered are assumed  to be
finite, simple and connected. For a graph $X$,
 $V(X)$, $E(X)$, $A(X)$ and  $\rm {Aut(X)}$  denote its vertex set, edge
set, arc set,  and full automorphism group, respectively.  For $u, v  \in  V(X)$,  $ \{u,v\} $ denotes  the
edge incident to $ u $ and $ v$ in $X$, and  $ N_{X}(u)$ denotes the
neighborhood of $u $ in $X$, that is, the set of vertices adjacent to
$ u $ in $X$.

A graph $\widetilde{X}$ is called a covering of a graph $X$ with
projection $p : \widetilde{X} \rightarrow X$ if there is
a surjection $ p : V( \widetilde{X})\rightarrow V(X)$ such that $
p|_{N_{\widetilde{X}}(\widetilde{v})} :
N{_{\widetilde{X}}}(\widetilde{v})\rightarrow  N_{X}(v) $ is a
bijection for any vertex $v  \in  V(X)$ and $\widetilde{v} \in
p^{-1}(v)$. A permutation group $G$ on a set $ \Omega $ is said to be
semiregular if the stabilizer $ G_{v} $ of $v$ in $G$ is trivial for
each $v  \in  \Omega $, and is  regular  if $G$ is transitive, and
semiregular. Let $K$ be a subgroup of $\rm {Aut(X)}$ such that $K$ is
intransitive on $V(X)$. The quotient graph $X/K$ induced by $K$ is defined
as the graph such that the set  $ \Omega $  of $K$-orbits in $V(X)$ is
the vertex set of $X/K$ and $B$, $C  \in    \Omega $  are adjacent if
and only if there exists a $u  \in  B$ and  $v \in C$ such that $
\{u,v\}  \in $  $E(X)$. A covering $ \widetilde{X} $ of $X$ with a
projection $p$ is said to be regular (or $K$-covering) if there is a
 subgroup $K$ of the automorphism group ${\rm Aut( \widetilde{X})}$  such that  $K$ is semiregular on both $V(\widetilde{X})$ and $E(\widetilde{X})$ and graph $X$ is isomorphic to the quotient graph
$\widetilde{X}/K$, say by $h$, and the quotient map $ \widetilde{X}
 \rightarrow  \widetilde{X} /K$ is the composition $ph$ of $p$
and $h$. The group of covering transformations $\rm {CT(p)}$ of $p : \widetilde{X} \rightarrow X$ is the group of all self equivalences of $p$, that is, of all automorphisms $ \widetilde{\alpha } \in {\rm  Aut(  \widetilde{X})}$ such that $p= \widetilde{\alpha}  p$. If
$\widetilde{X}$ is connected, $K$ becomes the covering transformation
group.

 For a graph $X$ and a subgroup $G$ of $\rm {Aut(X)}$, $X$ is said to
be $G$-vertex-transitive, $G$-edge-transitive or $G$-arc-transitive  if $G$
is transitive on $V(X)$, $E(X)$ or $A(X)$,
respectively, and $G$-arc-regular if $G$ acts regularly on $A(X)$. A graph $X$ is called vertex-transitive, edge-transitive, arc-transitive, or arc-regular
if $X$ is $\rm {Aut(X)}$-vertex-transitive, $\rm {Aut(X)}$-edge-transitive,
$\rm {Aut(X)}$-arc-transitive, or $\rm {Aut(X)}$-arc-regular, respectively.  Let $X$ be a tetravalent $G$-half-arc-transitive graph for a subgroup $G$ of $\rm {Aut(X)}$, that is $G$ acts transitively on $V(X)$, $E(X)$, but not  $A(X)$.
Then under the natural action of $G$ on $V (X) \times V (X)$, $G$ has two orbits on the arc set $A(X)$, say $A_{1}$ and $A_{2}$, where
$A_{2} = \{ (v,u)|(u,v) \in A_{1} \}$. Therefore, one may obtain two oriented graphs with the vertex set  $V (X)$  and  the arc sets $A_{1}$ and $A_{2}$. Assume that $D_{G}(X)$ be one of the two oriented graphs. Also in the special case, if $G=\rm{Aut(X)}$ then $X$ is said to be 1/2-transitive or half-arc-transitive.

By Tutte \cite{tut},   each connected  vertex-transitive and edge-transitive graph of odd valency is arc-transitive.  So  half-arc-transitive graphs of odd valency do not exist. Bouwer \cite{Bou} answered Tutte's question about existence of half-arc-transitive graphs of even valency.  A number of authors later studied the construction of these graphs. See, for example \cite{Ant, Als, con, Du, zhou, Huj, kuntar, kun, li, wa, wf, fff}.  Let $p$ be a prime.  There are no
half-arc-transitive graphs of order $p$, $p^{2}$ and $2p$ (see \cite{chao,ox}). Feng, Kwak, Wang
and Zhou \cite{feng} classified the connected tetravalent half-arc-transitive
graphs of order $2pq$  for distinct odd primes $p$ and $q$. The tetravalent half-arc-transitive graphs of order
$p^{5}$, $p^{4}$, $2p^{2}$, $p^{3}$ and $2p^{3}$  are classified in \cite{cheng, fng,  ff, xu, zz} respectively.  Wang  et al.   \cite{zw} studied
tetravalent half-arc-transitive graphs of order a product of three primes. In \cite{liu}, Liu studied tetravalent half-arc-transitive graphs of order $p^{2}q^{2}$ with $p$, $q$ distinct odd primes. Feng et al. \cite{wang} classified the tetravalent
half-arc-transitive graphs of order $4p$. In \cite{Cui} a complete classification  of tetravalent half-arc-transitive
metacirculants of order 2-powers was given. In \cite{wwl}, a classification  of all tetravalent half-arc-transitive graphs of
order $8p$ was given. In this paper, we will study tetravalent half-arc-transitive graphs of order $12p$.

\section{Preliminaries}
Let $X$ be a graph and $K$ be a finite group. By $a^{-1}$ we mean the
reverse arc to an arc $a$. A voltage assignment (or $K$-voltage
assignment) of $X$ is a function $\xi : A(X) \rightarrow  K$ with the
property that $\xi(a^{-1}) = \xi(a)^{-1}$ for each arc $a
\in A(X)$. The values of $\xi$ are called voltages, and $K$ is the
voltage group. The graph $ X  \times_{{\xi}} K$  derived from a
voltage assignment $\xi: A(X)\rightarrow K$ has vertex set $V(X)
\times K$ and edge set $E(X) \times K$, so that an edge $(e, g)$ of $X$
$\times$ K joins a vertex $(u, g)$ to $(v, \xi(a)g)$ for $a = (u, v)
\in A(X)$ and $g \in K$, where $e = \{u,v\}$. Clearly, the derived graph $ X \times_{{\xi}} K$ is a covering of $X$
with the first coordinate projection $ p :  X \times_{{\xi}} K
\rightarrow  X$, which is called the natural projection. By defining
$(u, g')^{g} = (u, g'g)$ for any $g \in K$ and $(u, g') \in
V(X \times_ {\xi} K)$, $K$ becomes a subgroup of $ {\rm Aut(X \times _{\xi} K)}$ which acts semiregularly on $V( X
\times_{{\xi}} K)$. Therefore, $ X \times_{{\xi}} K$ can be viewed
as a $K$-covering. For each $u \in V(X)$ and $\{u,v\} \in E(X)$, the vertex
set $\{(u, g) | g\in K\}$ is the fibre of $u$ and the edge set $\{(u,
g) (v,\xi(a)g) | g\in K\}$ is the fibre of $\{u,v\}$, where $a = (u, v)$. The group $K$ of automorphisms of $X$ fixing every fibre setwise
is called the covering transformation group.
Conversely, each regular covering $\widetilde{X}$ of $X$ with a
covering transformation group $K$ can be derived from a $K$-voltage
assignment.  Given a spanning tree $T$ of the graph $X$, a voltage assignment
$\xi$ is said to be $T$-reduced if the voltages on the tree arcs are
the identity. Gross and Tucker in \cite{Gross} showed that every regular
covering $\widetilde{X}$ of a graph $X$ can be derived from a $T$-reduced voltage assignment $\widetilde{X}$ with respect to an
arbitrary fixed spanning tree $T$ of $X$.

 Let $\widetilde{X}$ be a $K$-covering of $X$ with a projection $p$.
If $\alpha \in \rm {Aut(X)}$ and
$\widetilde{\alpha}\in \rm{Aut(\widetilde{X})}$ satisfy
$\widetilde{\alpha}p =p\alpha$, we call $\widetilde{\alpha}$ a
lift of $\alpha$, and $\alpha$ the projection of
$\widetilde{\alpha}$.  The lifts and projections of such subgroups are of
course subgroups in $\rm{Aut(\widetilde{X})}$ and $\rm {Aut(X)}$, respectively.

Let $G$ be a group, and let $S \subseteq G$ be a set of group elements such that the identity element $1$ not in $S$. The Cayley graph associated with $(G,S)$ is  defined as the graph having one vertex associated with each group element,  edges $(g,h)$ whenever $hg^{-1}$ in $S$. The Cayley graph $X$ is denoted by $\rm {Cay(G, S)}$.
In graph theory, the lexicographic product or (graph) composition $G[H]$ of graphs $G$ and $H$ is a graph such that the vertex set of $G[H]$ is the cartesian product $V(G)\times V(H)$; and
any two vertices $(x,y)$ and $(v,w)$ are adjacent in $G[H]$ if and only if either $x$ is adjacent with $v$ in $G$ or $v = x$ and $w$ is adjacent with $y$ in $H$. Clearly, if $G$ and $H$ are arc-transitive then $G[H]$ is arc-transitive.

Let $X$ be a tetravalent $G$-half-arc-transitive graph for some $G\leq  \rm {Aut (X)}$.
Then no element of $G$ can interchange a pair of adjacent vertices in $X$. By \cite{holt},  there is no half-arc-transitive graph with less then $27$ vertices. Half-arc-transitive graphs have  even valencies.   An even length cycle $C$ in $X$ is a
$G$-alternating cycle if every other vertex of $C$ is the head and every other vertex of
$C$ is the tail of its two incident edges in $D_{G}(X)$.  All $G$-alternating cycles in $X$ have the same length. The radius of graph is  half of the length of an alternating cycle.  Any two adjacent $G$-alternating cycles in $X$ intersect in the same number
of vertices, called the $G$-attachment number of $X$. The intersection of two adjacent
G-alternating cycles is called a $G$-attachment set.
 We say that $X$ is tightly attached if the attachment number of $X$ equal  with its radius. 
 
Now we introduce  graph $X(r;m,n)$ and a result due to Maru\v{s}i\v{c}.

Suppose that $m \geq 3$ be an integer, $n \geq 3$ an odd integer and let $r \in \mathbb Z^{*}_{ n}$ satisfy $r^{m} = \pm 1$. The
graph $X(r;m,n)$ is defined to have vertex set $V = \{u_{i} ^{j} | i \in \mathbb Z_{m},j \in \mathbb Z_{n}\}$ and edge set
$E = \{\{u_{i} ^{j} ,u_{i+1} ^{j\pm r^{i}}\} | i \in \mathbb Z_{m},j \in \mathbb Z_{n}\}$.

\begin{proposition}\label{1}
\emph{[\cite{mar}, Theorem 3.4]} A connected tetravalent graph $X$ is a tightly attached half-arc-transitive graph of odd radius $n$ if and only if $X \cong X(r;m,n)$, where $m \geq 3$, and $r \in \mathbb Z^{*}_{ n}$ satisfying $r^{m} = \pm 1$, and moreover none of the following conditions is fulfilled:\\
(1) $r^{2} = \pm 1$;\\
(2) $(r;m,n) = (2;3,7)$;\\
(3) $(r;m,n) = (r;6,7k)$, where $k \geq 1$ is odd, $(7,k) = 1$, $r^{6} = 1$, and there exists a
unique solution $q \in \{r,-r,r^{-1},-r^{-1}\}$ of the equation $x^{2} +x -2 = 0$ such that
$7(q - 1) = 0$ and $q \equiv 5 ~(\rm{mod} ~7)$.
\end{proposition}
The following is the main result of  the paper tetravalent half-transitive graphs of order $4p$.
\begin{proposition}\label{2}
\emph{[\cite{wang}, Theorem 3.3]}  Let $p$ be a prime and $X$ a tetravalent graph of order $4p$. Then, $X$ is half-transitive
if and only if $ p \equiv 1~ (\rm{mod} ~8)$ and $X \cong X(r;4, p)$(denote by $X(4, p)$ the graph $X(r; 4, p))$.
\end{proposition}
Now we express an observations about  tetravalent half-arc-transitive graphs.

\begin{proposition}\label{3}
\emph{[\cite{maru}, Lemma 3.5]} Let $X$ be a connected tetravalent $G$-half-arc-transitive graph for some $G \leq \rm {Aut(X)}$, and let $\Delta$ be a $G$-attachment set of $X$. If $|\Delta| \geq 3$, then the vertex-stabilizer of $v \in V (X)$ in $G$ is of order $2$.
\end{proposition}

\begin{proposition}\label{4}
\emph{\cite{gor}} A non-abelian simple group whose order has at
most three prime divisors is isomorphic to one of the following groups:
\begin{center}
$\rm {A_{5}, A_{6},PSL(2,7),PSL(2,8),PSL(2,17),PSL(3,3),PSU(3,3),PSU(4,2)}$,
 \end{center}
 whose orders are $2^{2}~.~ 3 ~. ~5$, $ 2^{3} ~.~ 3^{2} ~.~ 5$, $ 2^{3} ~.~ 3 ~.~ 7$, $ 2^{3} ~.~ 3^{2} ~.~ 7$, $2^{4} ~.~ 3^{2} ~.~ 17$, $ 2^{4} ~.~ 3^{3} ~.~ 13$, $
2^{5}~.~3^{3}~.~7$, $ 2^{6} ~.~3^{4}~.~5$, respectively.
\end{proposition}
The following result is  extracted from [\emph{\cite{Ber}}, Theorem 1].
\begin{proposition}\label{5}
 Let $X$ be a tetravalent arc-transitive graph of order $2pq$ where $p$ and $q$ are odd and distinct primes. Then one of the
following holds:\\
(1) $X$ is arc-regular and appears in \emph{\cite{zh}};\\
(2) $X$ is isomorphic to the lexicographic product $C_{pq}[2K_{1}]$ of the cycle $C_{pq}$ and the edgeless graph on two vertices $2K_{1}$.
\end{proposition}
In the following, we describe the structure of the graphs required in this paper [\cite{pot}, \cite{prag}, \cite{zhoo}].

The Rose Window graph $R_{6}(5,4)$   is a tetravalent graph with   $12$ vertices. Its vertex set is $\{ S_{i}, Q_{i} | i \in Z_{6} \}$. The graph has four kinds of edges: kind of edges: $S_{i}S_{i+1}$ (rim edges), $S_{i}Q_{i}$ (inspoke edges), $S_{i+5}Q_{i}$ (outspoke edges) and
$Q_{i}Q_{i+4}$ (hub edges). $|\rm{Aut(R_{6}(5,4))}|=48$. Fig 1 shows $R_{6}(5,4)$.

  A general Wreath graph  $W(6, 2)$ has $12$ vertices and it is  regular of valency 4. Its vertex set is $\{ E_{i}, F_{i} | i \in Z_{6} \}$, where $E_{i}= (i,0)$ and $F_{i}= (i,1)$. Its edges are $\{E_{i}, E_{i+1} \}$, $\{E_{i}, F_{i+1} \}$, $\{F_{i}, E_{i+1} \}$ and $\{F_{i}, F_{i+1} \}$. $|\rm{Aut(W(6, 2))}|=768$. See Fig 2.
 \begin{center}
   \includegraphics[width=0.3\textwidth]{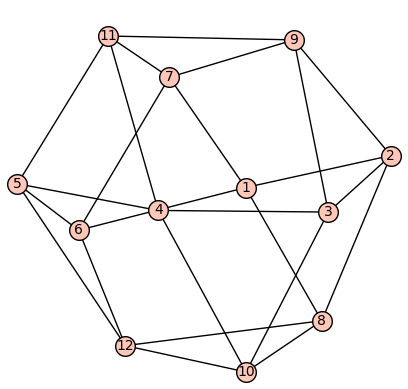}\hspace{1.5cm}\includegraphics[width=0.3\textwidth]{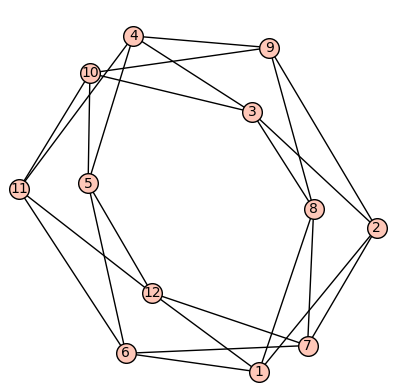}\\
   Fig 1. The Rose Window graph $R_{6}(5,4)$ \hspace{1.5cm} Fig 2. The Wreath graph  $W(6, 2)$
  \end{center}

\vspace{1cm}The  graph $C(2;p,2)$ was first defined by Praeger and Xu [\cite{prag},  Definition 2.1 (b)].
Let $p$ be an odd prime. The graph $C(2;p,2)$ has vertex set $\mathbb Z_{p} \times (\mathbb Z_{2} \times \mathbb Z_{2})$ and its edges are defined by
$\{(i,(x,y)),(i+1,(y,z))\} \in  E(C(2;p,2))$ for all $i \in  \mathbb Z_{p}$ and $x, y, z \in \mathbb Z_{2}$. $ \rm {Aut(C(2;p,2))}\cong   D_{2p} \ltimes  \mathbb Z_{2}^{p}$.

  Let   $ p \equiv 1~ (\rm{mod} ~4)$, where $p$ is a prime and $w$ is an element of order 4 in $\mathbb Z^{*}_{p}$. The graph $CA^{0}_{4p}$ is $ \rm{Cay}(G, \{ a, a^{-1}, a^{w^{2}}b, a^{-w^{2}}b \})$ and the graph $CA^{1}_{4p}$ is $ \rm{Cay}(G, \{ a, a^{-1}, a^{w}b, a^{-w}b \})$, where $G= <a>  \times <b> \cong \mathbb Z_{2p} \times \mathbb Z_{2} $.

\section{Main Results}
In this section,  we study all tetravalent half-arc-transitive graphs of order $12p$ where $p$ is a prime. To do this, we prove  the following results.
\begin{lemma}\label{9}
Let $X$ be a  graph, $G\leqslant \rm {Aut(X)}$, $N \trianglelefteq G$ and $X$ be  $N$-regular covering of $X_{N}$. Then  $X$ is  $G$-half-arc-transitive if and only if $X_{N}$ is  $G/N$-half-arc-transitive.
\end{lemma}
\begin{proof}
Suppose that $N \trianglelefteq G$ and $X$  is  $G$-half-arc-transitive. Since $X$ is   $N$-regular covering of $X_{N}$, it follows that $K=N$ and $G/N \leqslant \rm {Aut(X_{N})}$, where $K$ is the kernel of $G$ acting on orbits of $N$.
  Let   $x^{N}, y^{N}$ be two arbitrary vertices  of graph $X_{N}$. By our assumption there exits $g \in G$ such that $x^{g}=y$. Now $(x^{N})^{Ng}=(x^{g})^{N}= y^{N}$. It implies that $X_{N}$ is  $G/N$-vertex-transitive. Now, suppose that $\{x^{N}, y^{N}\}$ and $\{u^{N}, v^{N}\}$ are two arbitrary edges of $X_{N}$. Without loss of generality, we may suppose that $\{x, y\}$ and $\{u, v\}$ are two edges of $X$. By our assumption there exits $g \in G$ such that $\{x, y\}^{g}= \{u, v\}$.  Then we may assume that $x^{g}=u$ and $y^{g}=v$.
   Hence $(x^{N})^{Ng}=x^{Ng}=x^{gN}= u^{N}$ and $(y^{N})^{Ng}=y^{Ng}=y^{gN}= v^{N}$. Then $X_{N}$ is  $G/N$-edge-transitive. Suppose to contrary that $X_{N}$ is  $G/N$-arc-transitive. Let $(x,y)$ and $(u,v)$ are two arcs of graph $X$. Now $(x^{N},y^{N})$ and $(u^{N},v^{N})$  are two arcs of graph $X_{N}$. By our assumption, there exits $Ng \in G/N$ such that $(x^{N},y^{N})^{Ng}=(u^{N},v^{N})$. Therefore $(x^{N})^{Ng}= u^{N}$ and $(y^{N})^{Ng}= v^{N}$. Thus $x^{Ng}= u^{N}$ and $y^{Ng}= v^{N}$. Then $x^{g}= u^{n}$ and $y^{g}= v^{n'}$ for $n, n' \in N $ and so $(x,y)^{g}= (u^{n},v^{n'})$.
   There exits $n'' \in N $ such that $(u^{n},v^{n'})^{n''}=(u,v)$. Then  $(x,y)^{gn''}= (u^{n},v^{n'})^{n''} =(u,v)$. Therefore $X$ is $G$-arc-transitive, a contradiction. Then  $X_{N}$ is  $G/N$-half-arc-transitive.

Now suppose that $X_{N}$ is  $G/N$-half-arc-transitive. Thus $G/N$ acts transitively on $V(X_{N})$. Let $u,v \in V(X)$ and $u^{N}, v^{N} \in  V(X_{N})$. Then there is $Ng \in G/N$ such that $(u^{N})^{Ng}= v^{N}$ and hence, there is $n' \in N$ such that $u^{g}=v^{n'}$ and $u^{g(n')^{-1}}=v$. Then since $g(n')^{-1} \in G$, it implies that  $X$ is vertex-transitive. For any $\{ u,v \},\{x,y \} \in E(X)$, we have $\{ u^{N},v^{N} \},\{ x^{N},y^{N} \} \in E(X_{N})$. Since $X_{N}$ is $G/N$-edge-transitive, we have $Ng \in G/N$ such that $\{ u^{N},v^{N} \}^{Ng}=\{ x^{N},y^{N} \}$ and $\{ (u^{N})^{Ng},(v^{N})^{Ng} \}=\{ x^{N},y^{N} \}$. Without loss of generality, we may suppose that $(u^{N})^{Ng}=(u)^{Ng}= x^{N}$ and $(v^{N})^{Ng}=(v)^{Ng}= y^{N}$. There exits $n',n^{''} \in N$ such that $\{u, v\}^{g}= \{x^{n'},y^{n^{''}}\}$. Also there exits $n \in N$ such that   $\{x^{n'},y^{n^{''}}\}^{n}= \{x,y\}$.  Thus we may assume that  $\{u, v\}^{gn}= \{x,y\}$ and so $X$ is $G$-edge-transitive. Similar to the previous, it can be shown that if  $X_{N}$  is not $G/N$-arc-transitive then $X$ is not $G$-arc-transitive. Therefore $X$ is $G$-half-arc-transitive.
  \end{proof}
  The following lemma is basic for the  main result.
\begin{lemma}\label{10}
Let $X$ be a  half-arc-transitive graph,  $p$ is a prime and $N  \trianglelefteq \rm {Aut(X)}$, where $N \cong \mathbb Z_{p}$. If the quotient graph $X_{N}$ is a Cayley graph and has  the same valency with $X$ then $X$ is a $N$-regular covering of $X_{N}$ and $X$ is a Cayley graph.
\end{lemma}
\begin{proof}
Let $N$ be a normal subgroup of $A:=\rm {Aut(X)}$ and  $X_{N}$ be the quotient graph of $X$ with respect to the orbits of $N$ on $V(X)$. Assume that  $K$ is the kernel of $A$ acting on $V(X_{N})$. The stabilizer $K_{v}$ of $v \in V(X)$ in $K$ fixes the neighborhood of $v$ in $X$. The
connectivity of $X$ implies $K_{v}=1$ for any $v \in V(X)$ and hence  $N_{v}=1$.  If $N_{\{ \alpha,\beta\}} \neq 1$ then $N_{\{ \alpha,\beta\}}=N$. Since $X$ is  connected, there is a $\{ \beta, \gamma\} \in E(X)$ where $\beta,\gamma \in V(X)$. Then we have $g \in A$ such that $\{ \alpha,\beta\}=\{ \beta, \gamma\}^{g}$ because $X$ is an edge-transitive graph.  Hence $N_{\{ \alpha,\beta\}}=N_{\{ \beta, \gamma\}^g}= g^{-1}N_{\{ \beta, \gamma\}}g= N_{\{ \beta, \gamma\}}$. It is a contradiction and so $N_{\{ \alpha,\beta\}}=1$. Therefore $X$  is a $\mathbb Z_{p}$-regular covering of $X_{N}$. Now we prove that $X$ is a Cayley graph.
Let $X_{N}\cong \rm{Cay(G,S)}$, $X\cong X_{N} \times_{\xi}\mathbb Z_{p}$ where $\xi$ is the $T$-reduced voltage assignment  and $\tilde{G}$ is a lift of $G$ such that $\tilde{\alpha}p=p\alpha$ where $p: X \rightarrow X_{N}$ is regular covering projection, $\alpha \in \rm {Aut(X_{N})}$ and $\tilde{\alpha} \in A$.  For any $(x,k) , (y,k') \in V(X)$ where $k, k' \in \mathbb Z_{p}$ and $x,y \in V(X_{N})$, we have $\alpha \in \rm {Aut(X_{N})}$ such that $x^{\alpha}=y$. For $k^{''} \in \mathbb Z_{p}$, $(x,k)^{\tilde{\alpha}p}=(z, k^{''})^{p}=z$ where $(x,k)^{\tilde{\alpha}}=(z, k^{''})$. Also $(x,k)^{p\alpha} =x^{\alpha}=y$. Then $y=z$ and hence $(y,k), (y, k^{''}) \in p^{-1}(y)$. Therefore $\tilde{G}$ is transitive on $V(X)$. Now, we prove that $\tilde{G}$ is semiregular. Suppose that $(x,k)^{\tilde{\alpha}}=(x,k)$.   Now, since $G$ is semiregular and $\tilde{\alpha}p=p\alpha$, it implies that $x=(x,k)^{\tilde{\alpha}p}=(x,k)^{p\alpha}=x^{\alpha}$. Then $\alpha=1$ and hence $\tilde{\alpha}p=p$. Therefore $\tilde{\alpha} \in \rm{CT(p)}=\mathbb Z_{p}$ and since $\rm{CT(p)}$ is semiregular, it  follows that $\tilde{\alpha}=1$.
\end{proof}
By \cite{pot}, all tetravalent  half-arc-transitive graphs of order $12p$  where $p \leq 53 $ is a prime, are classified. Then in the following, we may assume that $p > 53$.
\begin{lemma}\label{11}
Let $X$ be a tetravalent half-arc-transitive graph of order $12p$, where $p$ is a prime. Then $\rm {Aut(X)}$ has a normal Sylow $p$-subgroup or $X$ is $\mathbb Z_{3}$-regular covering of $C(2;p,2)$ or $C_{2p}[2K_{1}]$.
\end{lemma}
\begin{proof}
Let $X$ be a tetravalent half-arc-transitive graph of order $12p$ where p is a prime.  Let $A:=\rm {Aut(X)}$.  Since the stabilizer $A_{v}$ of $v \in V(X)$ is a 2-group, we have $|A|= 2^{m+2}.3.p$, for some nonegative integer $m$.
  Suppose to the contrary that $A$ has no normal Sylow p-subgroups. Let $N$ be a minimal normal subgroup of $A$. We claim that $N$ is solvable. Otherwise,   by Proposition \ref{4} and since $p > 53$, we get a contradiction.   Then  $N$ is solvable   and hence it is an elementary abelian 2-,3- or p-group. \\
\textbf{Case I:}  $N$ is a 2-group.

Let  $X_{N}$ be the quotient graph of $X$ corresponding to the orbits of $N$ on $V(X)$. Then $|V(X_{N})|= 6p$ or $3p$.\\
\textbf{Subcase 1:} $|V(X_{N})|= 6p$.

Since $X$ is edge-transitive, $X_{N}$ has valency $2$ or $4$. Suppose  that $X_{N}$ has valency $2$. Then $X\cong C_{6p}[2K_{1}]$, which is arc-transitive. It is a contradiction. Assume now that $X_{N}$ has valency $4$. If $X_{N}$ is half-arc-transitive then by  [\cite{feng}, Theorem 4.1], $|\rm{Aut(X_{N})}|= 2^{2}.3.p$. Let $K$ be the kernel of $A$ acting on $V(X_{N})$. Since $K$ fixes each orbit of $N$,  the stabilizer  $K_{v}=1$ for any $v \in V(X)$. Then $|N|=|K|$. On the other hand $A/K \leqslant \rm{Aut(X_{N})}$. Since $A/K$ acts transitively on $V(X_{N})$ and $E(X_{N})$, $|A|=24p$. Then $1+np$ $|$ $24$. Since $p > 53$ then $P \trianglelefteq A$, a contradiction.
Now, suppose that $X_{N}$ is arc-transitive. Let $X_{N}$ has valency $4$. By Proposition \ref{5}, if $X_{N}$ is arc-regular then $|\rm{Aut(X_{N})}|= 24p$. By lemma \ref{9}, $A/K$ is half-arc-transitive and hence $|A|=24p$. Then $P \trianglelefteq A$ because $p > 53$ . It is a contradiction. If $X_{N}$ not be arc-regular then by Proposition \ref{5}, $Y=X_{N}\cong C_{3p}[2K_{1}]$ and $B= \rm {Aut(Y)}$.  $|B|= 2^{3p+1}.3.p$.
 Assume that  $M$ is a minimal normal subgroup of $B$. By the same argument as in the first paragraph, $M$  is solvable and hence it is an elementary abelian 2-,3- or p-group. First, assume that $M$ is a 2-group and   $Y_{M}$ is the quotient graph of $Y$ corresponding to the orbits of $M$ on $V(Y)$.  The quotient graph $Y_{M}$ has order $3p$ and valency $2$ or $4$. If $Y_{M}$ has valency  $4$ then $M_{v}=1$ for $v \in V(Y)$. Assume that $K_{1}$ be the kernel of $B$ acting on $V(Y_{M})$.  Hence $|K_{1}|=|M|$. Thus $B/K_{1} \leqslant \rm{Aut(Y_{M})}$.  It is a contradiction because      $\rm {|Aut(Y_{M})}|= 12p$ by [\cite{my},  Theorem 5]. If $Y_{M}$ has valency  $2$ then  $Y_{M} \cong C_{3p}$ and  $\rm {Aut(Y_{M})} \cong D_{6p}$.  Since $|K_{1}| \leq  2$, we have $|B| \leq 12p$. We get a contradiction because $p > 53$. Now,  suppose that $M$ be a 3-group. Then $|V(Y_{M})|= 2p$. Since $M_{v}=1$ for $v \in V(Y)$ by using  [\cite{gard}, Theorem 1.1(4)], $Y_{M}$ has valency $4$. By [\cite{chao}, Table 1], $Y_{M} \cong G(2,p,r)$ or $G(2p,r)$.  Then $|K_{1}|=|M|$ and hence $B/K_{1} \leqslant \rm{Aut(Y_{M})}$. It is a contradiction because $|\rm {Aut(Y_{M})}|= 2^{p+1}.p$ or $8p$ and $p > 53$. Let $M$ be a p-group. Then $|Y_{M}|= 6$. Since $M_{v} \leq M$ we have $|M_{v}|=1$. By [\cite{gard}, Theorem 1.1(4)], $Y_{M}$ has valency $4$. By \cite{pot},  $|\rm{Aut(Y_{M})}|= 48$. Hence $B/K_{1} \leqslant \rm{Aut(Y_{M})}$. It is a contradiction.
\\
\textbf{Subcase 2:} $|V(X_{N})|= 3p$.

Let  $|V(X_{N})|= 3p$ and $X_{N}$ has valency $2$. Then $X\cong C_{3p}[2K_{1}]$. This leads to a contradiction. If  $X_{N}$ has valency $4$ and it  is half-arc-transitive then  by [\cite{Als}, Theorem 2.5], $|\rm{Aut(X_{N})}|= 6p$. Since $X_{N}$ is an edge-transitive graph, $6p~|~|A/K|~|6p$. Then $|A|=24p$ and hence $P \trianglelefteq A$. It is a contradiction. Suppose now that  $X_{N}$ is arc-transitive. By [\cite{my},  Theorem 5], $|\rm{Aut(X_{N})}|= 12p$. Then with the same arguments as before,  a contradiction can be obtained. \\
\textbf{Case II:} $N$ is 3-group.

If  $|V(X_{N})|= 4p$ and $X_{N}$ has valency $2$, then  $X_{N}\cong C_{4p}$ and hence $\rm{Aut(X_{N})}\cong D_{8p}$. Since $K=K_{v}N$ for any $v\in V(X)$ and  $K$ acts faithfully on $V(X)$, we have $K \leqslant S_{3}$ and hence $K_{v}\leq 2$. Then  $|A|\mid 48p $. Therefore $P \trianglelefteq A$ because according to assumption $p > 53$. This leads to a contradiction. Now let $|V(X_{N})|= 4p$  and $X_{N}$ has valency $4$. Then $X_{N}$ is arc-transitive or half-arc-transitive. By   [\cite{zhoo}, Table 1] and Proposition \ref{2},    $X_{N}\cong C(2;p,2)$, $C_{2p}[2K_{1}]$, $CA^{0}_{4p}$,  $CA^{1}_{4p}$  or $X(4,p)$. Let  $X_{N}\cong C(2;p,2)$ or $C_{2p}[2K_{1}]$.  Since $X_{N}$ has valency $4$, $N$ acts semiregularly on $V(X)$ and so $X$ is a $\mathbb Z_{3}$-regular covering of $C(2;p,2)$ or $C_{2p}[2K_{1}]$. Assume that $Y=X_{N}\cong  CA^{0}_{4p}$ or $CA^{1}_{4p}$ and $B= \rm{Aut(Y)}$. Since $|K|=|N|$, we have $A/K \leq B$ and hence $|A| \leq 48p$. Then $P \trianglelefteq A$.    Suppose that $Y=X_{N}\cong  X(4,p)$ and $B= \rm{Aut(Y)}$. Since $Y$ is half-arc-transitive, we have $|B|=  2^{m+2}.p$, for some nonegative integer $m$. Let  $M$ be a minimal normal subgroup of $B$. Thus $M$ is an elementary abelian 2- or p-group. First, assume that $M$ be a p-group and   $Y_{M}$ be the quotient graph of $Y$ corresponding to the orbits of $M$ on $V(Y)$. Then $|V(Y_{M})|= 4$. Since $Y$ is an edge-transitive graph and $M_{v}=1$ for $v \in V(Y)$,  $Y_{M}$ has valency $4$, a contradiction. Suppose that $M$ is a 2-group. Therefore $|V(Y_{M})|= 2p$ or $ p$ and $Y_{M}$ has   valency $2$ or $4$.\\
\textbf{Subcase 1:} $|V(Y_{M})|= 2p$.

If  $Y_{M}$ has   valency $2$ then $Y\cong C_{2p}[2K_{1}]$, which is arc-transitive.  Since $Y$ is half-arc-transitive, we get a contradiction. Suppose now that $Y_{M}$ has   valency $4$. By [\cite{ox}, Table 1], $Y_{M}\cong G(2p,4)$ or  $G(2,p,2)$. Assume that  $Y_{M}\cong G(2p,4)$. Since $(K_{1})_{v}=1$, $|B/ K_{1}| \leqslant 8p$ and hence $|A| \leq 48 p$. It is a contradiction  because $p > 53$. Suppose that $Y_{M}\cong G(2,p,2)$. Let $Z= Y_{M} \cong G(2,p,2)$ and $C= \rm{Aut(Z)}$. Let $H$ be a minimal normal subgroup of $C$ and let $Z_{H}$ be the quotient graph of $Z$ with respect to the orbits of $H$. Since $|C|=2^{p+1}.p$,   $H$ is 2- or p-group.  Assume that $H$ is a 2-group. Thus  $|Z_{H}|=p$ and $Z_{H}$ has valency $2$ or $4$. By [\cite{chao}, Theorem 3],  $|\rm{Aut(Z_{H})}|=2p$ or $4p$. Assume that $K_{1}$ be the kernel of $C$ acting on $V(Z_{H})$. If  $Z_{H}$ has valency $4$ then $|K_{1}|=|H|=2$ because  $|(K_{1})_{v}| =1$. Then $C/ K_{1} \leqslant 16p$ and hence $2^{p+1}\leq 8p$. We get a contradiction because $p > 53$.  If  $Z_{H}$ has valency $2$ then $|K_{1}|\leq 8$ because  $|(K_{1})_{v}| \leq 2$. Thus $C/ K_{1} \leqslant 16p$ and hence $2^{p+1}\leq 8p$, a contradiction can be obtained. Now, suppose that $H$ is a p-group. Then $|Z_{H}|=2 $ with valency $2$, $4$, a contradiction.
 \\
 \textbf{Subcase 2:} $|V(Y_{M})|= p$.

 If  $Y_{M}$ has   valency $4$ then by lemma \ref{10}, $Y$  is $\mathbb Z_{2}$-regular cover of $Y_{M}$ and $Y$ is a Cayley graph. But by \cite{wang}$,   X(4,p)$ is not a Cayley graph, a contradiction. Suppose that $Y_{M}$ has  valency $2$ and hence $Y_{M}\cong C_{p}$. Assume that $K_{1}$ is the kernel of $B$ acting on $V(Y_{M})$ and $(K_{1})_{v} = 1$. Then $B/K_{1} \leqslant \rm{Aut(Y_{M})}$ and so $|B| \leq 8p$. Therefore $|A| \leq 24p$ and hence $P \trianglelefteq A$ because $p > 53$. Then $(K_{1})_{v} \neq  1$. Let $V(Y _{M})=\{ \Omega_{0},\Omega_{1},\Omega_{2},..., \Omega_{p-1}\}$. The subgraph induced by any two adjacent orbits is either a cycle of length $8$ or a union of two cycles of length $4$. Suppose that $\langle \Omega_{i}\cup \Omega_{i+1} \rangle$ is an 8-cycle.  Thus $K_{1}$ acts faithfully on each $\Omega_{i}$ and hence  $(K_{1})_{v} \cong \mathbb Z_{2}$. It implies that $|K_{1}|=8$.  Since $M$ is transitive on each $\Omega_{i}$ and $(K_{1})_{v} > 1$, all edges
in the induced subgraph $\langle \Omega_{i}\cup \Omega_{i+1} \rangle$ have the same direction either from $\Omega_{i}$ to $\Omega_{i+1}$
or from $\Omega_{i+1}$ to $\Omega_{i}$ in the oriented graph $D_{B}(Y)$.  It follows that  $B/K_{1} \cong \mathbb Z_{p}$ and $|B| \leq 8p$. Therefore $|A| \leq 24p$ and hence $P \trianglelefteq A$ because $p > 53$. Assume that $\langle \Omega_{i}\cup \Omega_{i+1} \rangle$  is a union of two $4$-cycles. Let $\Omega_{i}= \{ u_{i}^{0}, u_{i}^{1}, u_{i}^{2}, u_{i}^{3}\}$ for any $i$ in $\mathbb Z_{p}$. Then $B$ has an automorphism $\alpha$ of order $p$  such that for any $i$ in $\mathbb Z_{p}$, $\Omega_{i}^{\alpha}=\Omega_{i+1}$. Let $(u_{i}^{j})^{\alpha}=u_{i+1}^{j}$ for $i$ in $\mathbb Z_{p}$ and $j$ in $\mathbb Z_{4}$. Consider a 4-cycle $C$ in  the induced subgraph $\langle \Omega_{0}\cup \Omega_{1} \rangle$ and let $n$ be the number of edges of $C$ which are in some orbit of $\alpha$. Then  $n=0,1$, or $2$. Consequently,  the induced subgraph $\langle \Omega_{0}\cup \Omega_{1} \rangle$ is one of the  of the following three cases.

 In the   Case $1$, $Y$ is disconnected, a contradiction.  In the   Case $2$, $Y\cong C_{2p}[2K_{1}]$.  We get a contradiction because  $Y\cong X(4,p)$. In the   Case $3$, $Y\cong C(2;p,2)$ that is arc-transitive. It is  a contradiction because $X(4,p)$ is a half-arc-transitive graph.

 \begin{center}
    \includegraphics[width=1\textwidth]{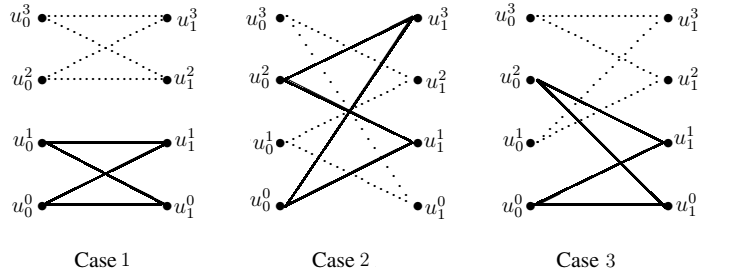}\\
   Fig 3. The induced subgraph $\langle \Omega_{0}\cup \Omega_{1} \rangle$.
  \end{center}
\textbf{Case III:} $N$ is p-group.

If $|N|=p $ then $N$ is a normal Sylow $p$-subgroup of $A$ as claimed.
\end{proof}
\begin{theorem}
Let $X$ be a connected  tetravalent vertex-transitive
and edge-transitive graph  of order $12p$, where $p >53$ is a prime. Then one of the following
statements holds:\\
(1) $X$ is half-arc-transitive if and only if $X\cong X(r;12,p)$ such that $r \in \mathbb Z^{*}_{ p}$ satisfying $r^{12} = \pm 1$.\\
(2) $X$ is half-arc-transitive Cayley graph  if and only if $X \cong Y \times_{{\xi}}  \mathbb Z_{p}  $, where $Y = W(6,2)$ or $ R_6(5,4)$ and $\xi : A(Y) \rightarrow  \mathbb Z_{p}$.\\
(3) if $X$ is half-arc-transitive then $X\cong C(2;p,2) \times_{{\xi}} \mathbb Z_{3}$, where $\xi : A(C(2;p,2)) \rightarrow  \mathbb Z_{3}$.\\
(4) if $X$ is half-arc-transitive then $X$ is a Cayley graph and $X\cong C_{2p}[2K_{1}] \times_{{\xi}} \mathbb Z_{3}$, where $\xi : A(C_{2p}[2K_{1}]) \rightarrow  \mathbb Z_{3}$.

\end{theorem}

\begin{proof}
Let $X$ be a tetravalent half-arc-transitive graph of order $12p$ and hence $|A|= 2^{m+2}.3.p$ for some integer $m\geq 0$. By Lemma \ref{11}, either $P \trianglelefteq A$  or $X$ is a $\mathbb Z_{3}$-regular covering of $C(2;p,2)$ or $C_{2p}[2K_{1}]$. First, suppose that $P \trianglelefteq A$. Now,  let $X_{P}$ be the quotient graph of $X$ corresponding to the orbits of $P$. Assume that $K$ is the kernel of $A$ acting on $V(X_{P})$. Then $V(X_{P})=12$ and $X_{P}$ has valency $2$ or $4$. If $X_{P}$ has valency $2$ then $X_{P}\cong C_{12}$ and hence $\rm{Aut(X_{P}})\cong D_{24}$. By Proposition \ref{3}, $A_{v}\cong \mathbb Z_{2}$ and hence $|A|=24p$. The attachment number of $X$ is equal to its radius. So $X$ is a tetravalent tightly attached half-arc-transitive graph of odd radius $p$.  By Proposition \ref{1}, $X\cong X(r;12,p)$ and  $|A|=24p$.  Also, by Proposition \ref{1},  it is trivial  that $ X(r;12,p)$ is  tetravalent half-arc-transitive graph of order $12p$.
Assume that  $X_{P}$ has valency $4$ and $X_{P}$ is arc-transitive or half-arc-transitive. There is no half-arc-transitive graph of order $12$. Suppose that $X_{P}$ is an arc-transitive graph. By \cite{pot}, $W(6,2)$ and $R_6(5,4)$ are the only   two arc-transitive graphs of order $12$.  These graphs  are  Cayley graphs by \cite{sage}. Since $P$ acts semiregular on $V(X)$ and $E(X)$, by Lemma \ref{10}, $X$ is a $\mathbb Z_{p}$-regular covering of $X_{P}$ and $X$ is a Cayley graph. For convenience,  consider the graphs  $W(6,2)$ and $R_6(5,4)$. By \cite{sage},  these graphs  have half-arc-transitive subgroups. By Lemma \ref{9}, since $X_{P}$ is $A/P$-half-arc-transitive then $X$ is half-arc-transitive. Now suppose that $X$ is half-arc-transitive which is  $\mathbb Z_{3}$-regular covering of $X_{P}$. By Lemma \ref{11} and Lemma \ref{10},   the cases $(3)$ and $(4)$ holds.

\end{proof}


\end{document}